\documentclass{amsart}

\usepackage{amsfonts,amsmath,amssymb,mathrsfs}
\usepackage{latexsym,amscd,graphicx,xcolor}

\newcommand{\ve}{\varepsilon}
\newcommand{\vp}{\varphi}
\newcommand{\ld}{\ldots}

\DeclareMathOperator*{\ot}{\otimes}
\DeclareMathOperator*{\op}{\oplus}

\newcommand{\beq}{\begin{equation}}
\newcommand{\eeq}{\end{equation}}
\newcommand{\beas}{\begin{eqnarray*}}
\newcommand{\eeas}{\end{eqnarray*}}

\newcommand{\cH}{\mathcal{H}}

\newcommand{\R}{\mathbb{R}}

\renewcommand{\H}{\mathbb{H}}

\DeclareMathOperator{\Span}{\mathrm{Span}}

\newtheorem{theorem}{Theorem}[section]

\newtheorem{lemma}[theorem]{Lemma}
\newtheorem{corollary}[theorem]{Corollary}
\newtheorem{proposition}[theorem]{Proposition}

\title[Graded polynomial identities]{Graded polynomial identities as identities of universal algebras}
\author{Yuri Bahturin}
\thanks{Supported by NSERC grant number 227060-14}
\address{Department of Mathematics and Statistics, Memorial University of Newfoundland, St. John's, NL, A1C5S7, Canada.}
\email{bahturin@mun.ca}
\author{Felipe Yasumura}
\thanks{Supported by Fapesp grant number 2017/11.018-9.}
\address{Department of Mathematics and Statistics, Memorial University of Newfoundland, St. John's, NL, A1C5S7, Canada\\ Department of Mathematics, Universidade Estadual de Campinas, Campinas, SP, 13083-859, Brazil.}
\email{fyyasumura@mun.ca}


\begin{document}
\begin{abstract}
Let $A$ and $B$ be finite-dimensional simple algebras with arbitrary signature over an algebraically closed field. Suppose $A$ and $B$ are graded by a semigroup $S$ so that the graded identitical relations of  $A$ are the same as those of $B$.  Then $A$ is isomorphic to $B$ as an $S$-graded algebra.
\end{abstract}
\maketitle

\section{Introduction}

In this manuscript, we treat the well-known question whether having the same set of polynomial identities guarantees the isomorphism of algebras.  Some obvious restrictions are necessary. In the non-simple case, we have that an algebra $A$ satisfies the same polynomial identities as the sum $A\op A$ of two copies of itself. If the field of coefficients is not algebraically closed, then there exist easy examples of non-isomorphic algebras satisfying the same set of ordinary polynomial identities; for example, the algebra of real quaternions $\H$ and the matrix algebra $M_2(\R)$ have the same polynomial identities but $\H\not\cong M_2(\R)$. 

So we need to restrict ourselves to the case of ``simple'' algebras over algebraically closed fields. Here being simple depends on the full set of structures on the algebras, for example graded-simple, involution-simple, differentially-simple and so on.

In the context of Lie algebras this question was settled by Kushkulei and Razmyslov \cite{KR}, and in the context of Jordan algebras by Drensky and Racine \cite{DR}, and Shestakov and Zaicev obtained the result for arbitrary simple algebras \cite{SZ2011}. The case of associative algebras is trivial thanks to Amitsur - Levitzki's Theorem.

The case of simple associative algebras graded by abelian group has been resolved  by Koshlukov and Zaicev \cite{KZ}. Having analyzed the structure of $G$-graded simple associative algebras, $G$ not necessarily abelian, Aljadeff and Haile managed to expand the result of Koshlukov and Zaicev to arbitrary groups \cite{AH}. Recently, Bianchi and Diniz studied the problem for arbitrary graded-simple algebras, where the grading is by an abelian group \cite{BD2018}.

Imposing various restrictions, this isomorphism question can be investigated in the case of other non-simple algebras. In \cite{DS} the authors study the question for certain types of gradings on the associative algebras of upper-block triangular matrices. Also, the classification of group gradings on upper triangular matrices \cite{VZ} together with the proof of the classification of elementary gradings on the same algebra \cite{DKV} leads to a positive answer in the context of graded associative algebras of upper triangular matrices.

\subsection{$\Omega$-algebras}\label{ssoa}
In this paper we deal with so called  $\Omega$-algebras, where $\Omega$ is a set, called \textit{signature}. One has $\Omega=\bigcup_{m=0}^\infty\Omega_m$. An algebra $A$ is called an $\Omega$-algebra, if $A$ is a vector space such that every $\omega\in\Omega_m$ defines an $m$-ary operation on $A$, that is, a linear map $\omega:\underbrace{A\otimes\cdots\otimes A}_\text{$m$ times}\to A$. In a natural way, one can define the standard notions of subalgebras, homomorphisms of algebras, ideals, and so on.

Given a  non-empty set $X$, one can define the free $\Omega$-algebra $F=F_\Omega(X)$ as follows.  First we build the set $W=W_\Omega(X)$ of  $\Omega$-monomials in $X$ as the union of subsets $W_n$, $n=0,1,2,\ld$ given by $W_0=\Omega_0\cup X$ and for $n>0$,
\[
W_n=\bigcup_{m=1}^\infty\bigcup_{\omega\in\Omega_m}\bigcup_{i_1+\cdots+i_m+1\le n} \omega(W_{i_1},\ld,W_{i_m}).
\] 
From this definition, it follows that for any $\omega\in \Omega_m$ and any $a_1,\ld,a_m\in W$ the expression $\omega(a_1,\ld,a_m)$ is a well-defined element of $W$. The elements of $W_n$ are called monomials of degree $n$.

Then we consider the linear span $F=F_\Omega(X)$ of $W=W_\Omega(X)$. If $F_n=\Span\{ W_n\}$ then $F=\bigoplus_{n=0}^\infty F_n$. The elements of $F_n$ are called homogeneous polynomials of degree $n$. In a usual way, one defines the degree of an arbitrary nonzero polynomial. By linearity, every $\omega\in\Omega_m$ defines an $m$-linear operation on $F$. Also, it follows that for any $\Omega$-algebra $A$ any map $\vp:X\to A$ uniquely extends to a homomorphism $\bar{\vp}: F_\Omega(X)\to A$. The $\Omega$-algebra $F=F_\Omega(X)$ is called the free $\Omega$-algebra with the basis (=set of free generators) $X$.

The equation of the form $f(x_1,\ld,x_n)=0$ where $f(x_1,\ld,x_n)\in F_\Omega(X)$ is called a (polynomial) identity in an $\Omega$-algebra $A$ if under any map $\vp: X\to A$ one has $\bar{\vp}(f(x_1,\ld,x_n))=0$. In other way, $f(a_1,\ld,a_n)=0$, for any $a_1,\ld, a_n\in A$. If $A=F_\Omega(X)$ then $I$ is a T-ideal if with every $f(x_1,x_2,\ld,x_n)\in I$ also $f(a_1,a_2,\ld,a_n)\in I$, for any $a_1,\ld,a_n\in F_\Omega(X)$. Given T-ideal $I$, the algebra $F_\Omega(X)/I$ is called a \textit{relatively free} algebra. The set of all (left hand sides) of the identitical relations in an algebra $A$, depending on the variables in $X$ is a T-ideal $T(A)$ of $F_\Omega(X)$. We denote the relatively free algebra $F_\Omega(X)/T(A)$ by $F_\Omega^A(X)$. If $A$ and $B$ satisfy the same identities then $F_\Omega^A(X)=F_\Omega^B(X)$. Finally, we denoted by $\text{var}_\Omega\,A$ the variety  of $\Omega$-algebras generated by $A$.

For now on, for each set $\Omega$, we assume that $\cup_{m=2}^\infty\Omega_m\ne\emptyset$. In his book \cite{R1994} Yuri Pitirimovich Razmyslov  proves the following remarkable result.

\begin{theorem}\cite[Corollary 1 of Theorem 5.3]{R1994}\label{Razmyslov}
Two simple finite-dimensional $\Omega$-algebras over an algebraically closed field, satisfying the same  polynomial identities, are isomorphic. In other words, in the variety $\text{var}_\Omega\,A$ generated by a simple algebra $A$ there are no other simple $\Omega$-algebras.
\end{theorem}

It is clear that this Theorem generalizes previous results for simple ungraded algebras. Moreover, Corollary 2 of Theorem 5.3 of \cite{R1994} states the same result for prime algebras.

In this paper  we show that any two algebras graded by the same group can be regarded as ungraded $\Omega$-algebras of the same
signature $\Omega$, and their graded identities are identities of $\Omega$-algebras. This allows us to settle the main problem in the case of finite-dimensional graded-simple (and graded-prime) algebras.

\section{Graded algebras as $\Omega$-algebras}

Let $K$ be an arbitrary field, $G$ any semigroup and $A$ an algebra over $K$, in the usual sense, that is with one binary operation. It is not necessary that $A$ is associative or is finite-dimensional. Consider a $G$-grading on $A$:
\[
A=\bigoplus_{g\in G}A_g.
\]

Then, for any $g\in G$ we have the natural projection $\pi_g:A\to A$
\begin{equation}\label{d_proj}
\pi_g\left(\sum_{h\in G}a_h\right)=a_g,\mbox{ where } a_h\in A_h, \mbox{ for any } h\in G.
\end{equation}
Consider the set $\Omega=\Omega_1\cup \Omega_2$, where $\Omega_1=\{ \pi_g\mid g\in G\}$ and $\Omega_2=\{ \mu\}$. We then view $A$ as an $\Omega$-algebra, whose underlying vector space is $A$ itself, $\mu$ a binary operation $\mu(x, y)=xy$, the original product on $A$, and each $\pi_g$ is a unary operation given by \eqref{d_proj}. The following two propositions are immediate.

\begin{proposition}\label{ideals}
Let $I\subset A$. Then $I$ is a $G$-graded ideal of $A$ if and only if $I$ is an ideal of $A$, as an $\Omega$-algebra.

In particular, $A$ is simple as a $G$-graded algebra if and only if $A$ is simple as an $\Omega$-algebra.
\end{proposition}

\begin{proposition}\label{morphism}
Let $A$ and $B$ be two $G$-graded algebras and let $\vp:A\to B$ be any map. Then $\vp$ is a homomorphism of $G$-graded algebras if and only if $\vp$ is a homomorphism of $\Omega$-algebras.

In particular, $A$ is isomorphic to $B$ as a  $G$-graded algebra if and only if $A$ is isomorphic to $B$ as an  $\Omega$-algebra.
\end{proposition}

Since we have only one binary operation in $\Omega$, there is no confusion to write $ab$ instead of $\mu(a,b)$ in all $\Omega$-algebras. It is then easy to prove that the following identities of $\Omega$-algebras hold in $A$ ($x,y\in X$):
\begin{enumerate}
\item[(i)] If $g,h,k\in G$, then
\begin{align*}
\pi_g(\pi_h(x)\pi_k(y))=\left\{\begin{array}{l}\pi_h(x)\pi_k(y),\text{ if $g=hk$},\\0,\text{ if $g\ne hk$}.\end{array}\right.
\end{align*}
\item[(ii)] Let $g,h\in G$, then
\begin{align*}
\pi_g(\pi_h(x))=\left\{\begin{array}{l}\pi_g(x),\text{ if $g=h$},\\0,\text{ if $g\ne h$}.\end{array}\right.
\end{align*}
\end{enumerate}

From now on, we will be assuming that the $G$-grading on $A$ is finite, that is, $\text{Supp}\,A=\{ g\in G\,|\, A_g\ne \{ 0\}\}$ is finite. This always holds if $A$ is finite-dimensional or if $G$ is finite. In this case, if $S=\text{Supp}\,A$, then the following are identities in $A$:

\begin{align*}
\mbox{ \text{(iii)} }x=\sum_{g\in S}\pi_g(x), \quad \pi_h(x)=0,\text{ for $h\notin S$}.
\end{align*}

Let $M=\bigcup_{k=1}^\infty M_{k}$ be a subset of the set of monomials in $W=W_\Omega(X)$ defined by induction, as follows. If $k=1$ then $M_1=\{ \pi_g(x)\,|\,g\in G,\,x\in X\}$. If $k>1$ then $w=w_1w_2$, where $w_1\in M_{\ell}$,  $w_2\in M_{k-\ell}$, with $\ell\ge 1$.

Now let $T$ be the T-ideal in $F_\Omega(X)$ defined by the identities (i), (ii) and (iii). 

\begin{lemma}\label{pol_gen}
$F_\Omega(X)$ is the vector space span of $M$ and $T$.
\end{lemma}
\begin{proof}

In other words, using transformations (i), (ii) and (iii), one can reduce any monomial $u\in W=W_\Omega(X)$ to a linear combination of monomials $w$ in $M$.  Actually, we also prove that $\pi_g(w)=w$ for an appropriate $g\in G$, for any $w\in M$. Let us call such monomials homogeneous.

Let us use induction by $\deg u$. A monomial of degree 0 is an $x\in X$. Using (iii), we can write $x=\sum_{g\in S}\pi_g(x)$, proving the claim. The monomials of degree 1 are $\pi_g x$, for some $g\in G$ and $x\in X$, and $xy$, for $x,y\in X$. The former ones are in $M_1$ by definition. The latter ones, using (iii), can be written as linear combinations of monomials $w=\pi_g(x)\pi_h(y)$. Again by definition, these are in $M$. In addition, by (i), $\pi_{gh}(w)=w$. Now assume $\deg u=k>1$. Then either $u=\pi_g(w)$, where $\deg w= k-1$ or $u=w_1w_2$, where $\deg w_i<k$ for $i=1,2$.  By induction, $w, w_1, w_2$ are linear combinations of monomials in $M$, that is, either $\pi_h(x)$, $h\in G$, $x\in X$, or $v_1v_2$, where $\pi_{g_1}(v_1)=v_1$, $\pi_{g_2}(v_2)=v_2$ are homogeneous monomials in $M$, $g_1,g_2\in G$. Then if $u=w_1w_2$, $u$ is a linear combination of monomials in $M$. Now $\pi_g(\pi_h(x))= \delta_{gh}\pi_g(x)$ by (i), 
\[\pi_g(v_1v_2)=\pi_g(\pi_{g_1}(v_1)\pi_{g_1}(v_1))=\delta_{g, g_1g_2}\pi_{g_1}(v_1)\pi_{g_1}(v_1)=\delta_{g, g_1g_2}v_1v_2.
\]
This proves that $u=\pi_g(w)$ is a linear combination of monomials in $M$ as well. Also, this latter calculation proves our claim that $\pi_{g_1g_2}(v_1v_2)=v_1v_2$.
\end{proof}

 Now we denote by $K\langle X^G\rangle$ the free non-associative $G$-graded algebra, with free generators $x^{(g)}$, where $x\in X$, $g\in G$. Consider the homomorphism $\bar{\psi}:K\langle X^G\rangle\to F_\Omega^A(X)$ extending the map
$\psi(x^{(g)})=\pi_g(x)$. By previous lemma, $\bar{\psi}$ is surjective.

\begin{lemma}\label{translation}
Let $f\in K\langle X^G\rangle$. Then $f=0$ is a graded identity of $A$ if and only if $\bar{\psi}(f)=0$ is a polynomial identity of $A$ as an $\Omega$-algebra.
\end{lemma}
\begin{proof}
Assume $f=0$ is not a $G$-graded identity of $A$. Then there exists an evaluation
\begin{align*}
e:x_i^{(g)}\mapsto a_i^{(g)}\in A_g,
\end{align*}
such that $e(f)\ne0$. Define the evaluation $e':x_i\mapsto\sum_{g\in G}a_i^{(g)}$ (since the grading is finite, the sum is well defined). Let $w=w(x_1^{(g_1)},\ld,x_n^{(g_n)})\in K\langle X^G\rangle$ be a monomial. Then $\bar\psi(w)=w(\pi_{g_1}(x_1),\ld,\pi_{g_n}(x_n))$. An easy induction by $\deg w$ proves that $e'(\bar\psi(w))=e(w)$, hence we obtain $e'(\bar\psi(f))=e(f)\ne0$, proving that $\bar\psi(f)$ is not a polynomial identity of $A$, as an $\Omega$-algebra.

Conversely, if $\bar\psi(f)$ is not a polynomial identity of $A$, as an $\Omega$-algebra, then there exists an evaluation $e_2':x_i\mapsto\sum b_i^{(g)}$ such that $e_2'(\bar\psi(f))\ne0$. So we can define the evaluation $e_2:x_i^{(g)}\mapsto b_i^{(g)}$, and the same argument shows that $e_2(f)=e_2'(\bar\psi(f))\ne0$, concluding that $f$ is not a $G$-graded polynomial identity of $A$.
\end{proof}

An immediate consequence is the following
\begin{corollary}\label{same_id}
Let $A$ and $B$ be two algebras endowed with finite $G$-gradings. Then $A$ and $B$ satisfy the same graded polynomial identities if and only if $A$ and $B$ satisfy the same polynomial identities as $\Omega$-algebras.
\end{corollary}

Now, using Proposition \ref{ideals} and Corollary \ref{same_id}, applying Razmyslov's Theorem  \ref{Razmyslov} and concluding with Proposition \ref{morphism}, we obtain our main result.

\begin{theorem}\label{grad_raz}
Let $A$ and $B$ be finite-dimensional $G$-graded algebras which are graded simple over an algebraically closed field $K$, where $G$ is any semigroup. Then $A$ and $B$ are isomorphic as $G$-graded algebras if and only if they satisfy the same $G$-graded polynomial identities.
\end{theorem}

\section{Graded $\Omega$-algebras}

Let $A$ be an $\Omega$-algebra and let $G$ be any semigroup. Consider a vector space $G$-grading on $A$, that is, we fix a vector space decomposition $A=\oplus_{g\in G}A_g$. For any $m\in\mathbb{N}$, one obtains naturally a $G$-grading on
\begin{align*}
\otimes^mA:=\underbrace{A\otimes\cdots\otimes A}_\text{$m$ times}
\end{align*}
defining the homogeneous component of degree $g$ by
\begin{align*}
\left(\otimes^mA\right)_g=\sum_{g_1g_2\cdots g_m=g}A_{g_1}\otimes\cdots\otimes A_{g_m},
\end{align*}
see \cite[chapter 1, p. 11]{EK2013}. We say that $A$ is a $G$-graded $\Omega$-algebra if every $m$-ary operation $\omega:\otimes^mA\to A$ is $G$-homogeneous, that is, $\omega\left(\left(\otimes^mA\right)_g\right)\subset A_g$. This notion generalizes the common notion of a semigroup grading on an algebra.

We can turn any $G$-graded $\Omega$-algebra $A$ into an $\Omega_G$-algebra with signature $\Omega_G=\Omega\cup\{\pi_g\mid g\in G\}$, as we did earlier in this paper. This gives a full faithful functor from the category of $G$-graded $\Omega$-algebras   to the category of $\Omega_G$-algebras.

Similarly to the previous case, we obtain
\begin{lemma}
Let $A$ and $B$ be two $G$-graded $\Omega$-algebras.
\begin{enumerate}
\item $A$ is simple as $\Omega_G$-algebra if and only if $A$ is simple as $G$-graded $\Omega$-algebra.
\item $A$ is isomorphic to $B$ as $\Omega_G$-algebras if and only if $A$ is isomorphic to $B$ as $G$-graded $\Omega$-algebras.
\end{enumerate}
\end{lemma}

Denote by $F_{\Omega_G}$ the free $\Omega_G$-algebra with a set $X$ of free generators. The following are polynomial identities of a given $G$-graded $\Omega$-algebra:
\begin{enumerate}
\renewcommand{\labelenumi}{(\roman{enumi})}
\item $\pi_{g}(\pi_{h}(x))=0$, if $g\ne h$,
\item $\pi_g(\pi_g(x))=\pi_g(x)$,
\item $\omega(\pi_{g_1}x_1,\ldots,\pi_{g_n}x_n)=\pi_g(\omega(\pi_{g_1}x_1,\ldots,\pi_{g_n}x_n))$, where $g=g_1\cdots g_n$, $\omega\in\Omega$.
\end{enumerate}
Moreover, if the $G$-grading is finite, then the following is also an identity:
\begin{enumerate}
\renewcommand{\labelenumi}{(\roman{enumi})}
\setcounter{enumi}{3}
\item $x=\sum_{g\in G}\pi_g x$ and $\pi_h(x)=0$, if $h\notin\text{Supp}\,A$.
\end{enumerate}

Let $T$ be the T-ideal generated by the identities (i)--(iv). Let $Y=\{\pi_g(x)\mid x\in X,g\in G\}$. Consider the set of monomials $M=W_\Omega(Y)$, which consists of all monomials generated by $Y$ using the operations of $\Omega$.

\begin{lemma}
The sets $M$ and $T$ span $F_{\Omega_G}$.
\end{lemma}
\begin{proof}
Let $w\in F_{\Omega_G}$ be a monomial. We work modulo $T$. We prove by induction on the degree of $w$ that
\begin{enumerate}
\renewcommand{\labelenumi}{(\alph{enumi})}
\item $w$ is a linear combination of monomials in $M$,
\item if $w\in M$, then there exists $g\in G$ such that $\pi_g(w)=w$.
\end{enumerate}
If $\deg w=0$, then we use identities (iv) and (ii) to conclude (a) and (b). So assume $\deg w>0$. Then we can write $w=\omega(w_1,\ldots,w_m)$, where $m\ge 1$, and $\deg w_i<\deg w$ for all $i$; or $w=\pi_g(u)$, for some $g\in G$ and $\deg u=\deg w-1$. For the first case, by induction step, every $w_i$ is a linear combination of monomials in $M$. Hence, $w$ is a linear combination of monomials in $M$. In addition, suppose $w_i\in M$ and $\pi_{g_i}w_i=w_i$, for some $g_i\in G$, for each $i$. Then identity (iii) implies that $w$ satisfies (b).

For the last case $w=\pi_g(u)$, we can apply induction step on $u$. Thus $u$ is a linear combination of monomials in $M$, say $u=\sum u_i$. Induction hypothesis says that $\pi_{h_i}u_i=u_i$, for some $h_i\in G$. By identities (i)--(ii), one has $\pi_g\pi_{h_i} u_i=\delta_{gh_i}u_i$. Hence, $w$ satisfies (a). If it happens that $w\in M$, then $u\in M$, and the last computation implies that $w$ satisfies (b) as well.
\end{proof}

Let $X^G=\{x^{(g)}\mid x\in X,g\in G\}$. For each $\omega_0\in\Omega_0$, associate an arbitrary homogeneous degree to it, $\deg_G\omega_0\in G$. Also, write $\deg_G x^{(g)}=g$. Then $F_\Omega^\text{gr}$, the free $\Omega$-algebra with the basis $X^G$, induces a $G$-grading as follows. We already defined the homogeneous degree for the elements of degree 0. Now, given $\omega(w_1,\ldots,w_n)$, where $\omega\in\Omega$, $n\ge1$, we set $\deg_G\omega(w_1,\ldots,w_n)=\deg_G w_1\cdots\deg_G w_n$. This is a well defined $G$-grading on $F_\Omega$.

There is no doubt that we should not exclude $0$-ary operations in the free $\Omega$-algebra since polynomial identities with 1 or without 1 play an essential role in the theory. However, it is interesting to mention one example. In the context of associative algebra $A$ with a unit (where $1$ is a 0-ary operation), it is not possible to find a graded homomorphism $F_\Omega\to A$, unless we imposed $\deg_G 1=1\in G$ in $\Omega_0$.

We can then consider the $G$-graded evaluations and speak about the $G$-graded polynomial identities of $G$-graded $\Omega$-algebras. Let $F_{\Omega}^\text{gr}(A)$ be a relatively free $G$-graded $\Omega$-algebra. Then the same argument as in the previous section can be used to conclude that

\begin{lemma}
The map $\psi_G:F_{\Omega}^\text{gr}(A)\to F_{\Omega_G}(A)$, given by $\psi_G(x_i^{(g)})=\pi_g(x_i)$ is a bijective homomorphism of $\Omega$-algebras. Moreover, $\psi_G(\text{Id}_\Omega^\text{gr}(A))=\text{Id}_{\Omega_G}(A)$.
\end{lemma}

Hence $A$ and $B$ satisfy the same $G$-graded polynomial identities as $\Omega$-algebras if and only if they satisfy the same polynomial identities as $\Omega_G$-algebras. As a consequence, we can apply Razmyslov's Theorem in the setting of $G$-graded $\Omega$-algebras.
\begin{corollary}\label{Universal_grad}
Let $A$ and $B$ be two finite-dimensional $G$-graded $\Omega$-algebras over an algebraically closed field, which are simple as $G$-graded $\Omega$-algebra, where $G$ is any semigroup. Then $A$ is isomorphic to $B$, as a $G$-graded $\Omega$-algebra, if and only if they satisfy the same $G$-graded polynomial identities as $G$-graded $\Omega$-algebras.
\end{corollary}

Applying \cite[Corollary 2 of Theorem 5.3]{R1994} instead of \cite[Corollary 1 of Theorem 5.3]{R1994}, we obtain the same statement of the previous lemma but for prime algebras.
\begin{corollary}\label{Universal_prime_grad}
Let $A$ and $B$ be two finite-dimensional $G$-graded $\Omega$-algebras over an algebraically closed field, which are \emph{prime} as $G$-graded $\Omega$-algebra, where $G$ is any semigroup. Then $A$ is isomorphic to $B$, as a $G$-graded $\Omega$-algebra, if and only if they satisfy the same $G$-graded polynomial identities as $G$-graded $\Omega$-algebras.
\end{corollary}

\section{Further applications of Razmyslov's Theorem}
Recently, a great number of researches is done concerning polynomial identities, codimension growth and its asymptotics, exponent and related properties of algebras with additional structure. To cite a few examples, see \cite{AGK2017,BL1998,BZ1998,DKL2006,DK2011,GIL2016,GMV2017,G2013,GK2014,IL2017,IM2018}. It turns out that all these algebras can be described in terms of gradings on some $\Omega$-algebras, as we will illustrate bellow.

\subsection{Algebras with involution}
Let $(A,\ast)$ be an algebra with involution. Let $\Omega_\ast=\{\mu,\ast\}$, where $\mu$ defines the original algebra product on $A$ and $\ast$ defines the involution on $A$. Thus an algebra with involution is an $\Omega_\ast$-algebra. Let $F^\ast(X)$ be the relatively free algebra satisfying the identities $x^{\ast\ast}=x$ and $(xy)^\ast=y^\ast x^\ast$.

In the classical theory, the free algebra with involution is defined as the free algebra $K\langle X\rangle$, together with applications of a symbol $\ast$. It turns out that an equivalent description is considering the free algebra $K\langle Y^+,Y^-\rangle$, with free generators $Y^+=\{y^+_i\}$ and $Y^-=\{y^-_i\}$, where we assume the variables in $Y^+$ and $Y^-$ symmetric and skew-symmetric with respect to $\ast$, respectively. The classical notion of free algebra with involution coincides with $F^\ast(X)$. In particular, the polynomial identities with involution of an algebra with involution coincides with the polynomial identities as $\Omega_\ast$-algebra. Hence, Razmylov's Theorem \ref{Razmyslov} (or Corollary \ref{Universal_grad}) translates as follows.

\begin{theorem}
Two finite-dimensional algebras with involutions over an algebraically closed field, which are involution-simple, satisfying the same polynomial identities with involutions are isomorphic as algebra with involution.
\end{theorem}


\subsection{Superalgebra with involution and superinvolution}
Superalgebra is an example of $\mathbb{Z}_2$-graded algebra. Given $(A,\ast)$ a superalgebra with involution (or superinvolution), we can consider $\Omega_{\ast,2}$, containing a binary product, the $\mathbb{Z}_2$-grading structure and $\ast$. Superalgebra with involution (or superinvolution) can be viewed as an $\Omega_{\ast,2}$-algebra. Using the same consideration made in the previous example, one concludes that the graded polynomial identities with involution of $A$ coincides with the graded polynomial identities as $\Omega_\ast$-algebra. Hence, Corollary \ref{Universal_grad} translates as follows.

\begin{theorem}
Two finite-dimensional superalgebras with involution (or superinvolution) over an algebraically closed field, which are simple as superalgebra with involution (or superinvolution), satisfying the same graded polynomial identities with involution (or superinvolution) are isomorphic as graded algebra with involution (or superinvolution).
\end{theorem}





\subsection{Colour Lie superalgebra}
Let $L=\bigoplus_{g\in G}L_g$ a $G$-graded algebra with product $[\cdot,\cdot]$ such that there exists an alternating bicharacter $\ve:G\times G\to K^\ast$ satisfying, for $x_i\in L_{g_i}$, $i=1,2,3$:
\begin{align*}
&[x_1,x_2]=-\ve(g_1,g_2)[x_2,x_1],\\
&[x_1,[x_2,x_3]]=[[x_1,x_2],x_3]+\ve(g_1,g_2)[x_2,[x_1,x_3]].
\end{align*}
Then $(L,\ve)$ is called a colour Lie superalgebra (see, for instance, \cite{BMPZ}). We can view colour Lie superalgebras as non-associative graded algebras, hence colour Lie superalgebras is a particular example of $\Omega$-algebras.

Another related example, given that $\text{char}\,K=p>0$, is a colour Lie $p$-superalgebra. A colour Lie $p$-superalgebra is a colour Lie superalgebra $L$ with an additional partial map $x\mapsto x^{[p]}$, defined on some homogeneous components, satisfying the following:
\begin{align*}
&(\alpha x)^{[p]}=\alpha^p x^{[p]},\\
&(\text{ad}\,x^{[p]})(z)=[x^{[p]},z]=(\text{ad}\,x)^p(z),\\
&(x+y)^{[p]}=x^{[p]}+y^{[p]}+\sum_i s_i(x,y),
\end{align*}
where $s_i$ is some polynomial. Note that, $x\mapsto x^{[p]}$ is not always linear. Hence, we cannot always see the ``raising to $p$-th power" as an unary operation. So, it is not obvious how we can describe a colour Lie $p$-superalgebra as an $\Omega$-algebra. However, in the context of simple algebras, the second identity completely defines $\text{ad}\,x^{[p]}$. Since $\text{ad}$ is an isomorphism for a simple finite-dimensional algebra, we conclude that the $p$-th power map is completely defined by the product of the algebra. In this way, an isomorphism of colour Lie $p$-superalgebras preserving the product will preserve the $p$-th power as well. Hence, Theorem \ref{grad_raz} (or Corollary \ref{Universal_grad}) reads as follows.

\begin{theorem}
Two finite-dimensional colour Lie ($p$-)superalgebras over an algebraically closed field, which are simple as colour Lie superalgebras, satisfying the same graded polynomial identities are isomorphic as colour Lie ($p$-)superalgebras.
\end{theorem}

\subsection{Trace identities}
Another important example to consider is trace identities. We consider the signature $\Omega_\text{tr}=\{\text{tr},\mu\}$, where $\mu$ is a binary operation and $\text{tr}$ is an unary operation. For a matrix algebra $M_n$, the signature $\Omega_\text{tr}$ consists of the following operations: $\mu$ is the usual matrix multiplication and $\text{tr}(A)=a\cdot I$, where $a$ is the usual trace of $A$ and $I$ is the $n\times n$ identity matrix.

The classical free algebra with trace is defined as the relatively free algebra with signature $\Omega_\text{tr}$ satisfying the following polynomial identities (see \cite{Raz1974}):
\begin{itemize}
\item $\text{tr}(x)y=y\text{tr}(x)$,
\item $\text{tr}(xy)=\text{tr}(yx)$,
\item $\text{tr}(x\text{tr}(y))=\text{tr}(x)\text{tr}(y)$.
\end{itemize}
So, Razmyslov's Theorem \ref{Razmyslov} (or Corollary \ref{Universal_grad}) translates as:

\begin{theorem}
Two finite-dimensional algebras over an algebraically closed field, which are simple as algebra with trace, satisfying the same trace polynomial identities are isomorphic as algebras with trace.
\end{theorem}

Traces of generic matrices are related to invariants of matrix algebras. Also, it is known that trace polynomial identities of matrix algebras are consequences of the Cayley-Hamilton identity \cite{Raz1974,Pro1976}. It is worth mentioning that \textit{ordinary} polynomial identities for the matrix algebra $M_n$ are known only when $n\le2$, for infinite fields.

Let $A_1=M_{n_1}\oplus\cdots\oplus M_{n_r}$ and $A_2=M_{n_1'}\oplus\cdots\oplus M_{n_s'}$, and assume $n_1\ge n_2\ge\ldots\ge n_r$, $n_1'\ge\ldots\ge n_s'$, $r>1$ and $s>1$. It is clear that $A_1\cong A_2$ if and only if $r=s$ and $n_i=n_i'$ for all $i$. Moreover, assume $n_1=n_1'$. In this case, $A_1$ and $A_2$ satisfy the same polynomial identities, namely, the polynomial identities of the matrix algebra $M_{n_1}$. Both are not simple as ordinary algebras, but they are $\text{tr}$-simple, if we define trace as the induced trace from $M_n$, where $n=n_1+\ldots+n_r$. Our results say that we can find a trace identity satisfied by one algebra, but not by the other.

\subsection{Algebras with the action of Hopf algebras}
Let $A$ be an algebra (with a unique binary operation) and $H$ a Hopf algebra. We say that $A$ is a left $H$-algebra, or an $H$-module algebra, if $A$ is an unital left $H$-module and for any $a,b\in A$ and $g,h\in H$ the following hold (see \cite{BL1998}):
\begin{itemize}
\item $(gh)\ast a=g\ast(h\ast a)$,
\item $h\ast(ab)=\sum(h_{(1)}\ast a)(h_{(2)}\ast b)$.
\end{itemize}

Algebras with Hopf actions include important examples. We cite two of them.
\begin{description}
\item[Action by automorphisms] Let $G$ be a subgroup of the group of automorphisms of the algebra $A$. It is well-known that the group algebra $KG$ is a Hopf algebra. Then the action of $G$ on $A$ is a particular case of Hopf action by the group algebra $KG$.
\item[Action by derivations] Let $D$ be a Lie subalgebra of the algebra of derivations of $A$. Then the universal enveloping algebra $U(D)$ is a Hopf algebra. Thus the action of $D$ on $A$ can be viewed as a Hopf action of $U(D)$.
\end{description}

Now we present the classical construction of the free Hopf module algebras (see \cite{BL1998}). Fix a Hopf algebra $H$. Let $T=T(H)=\sum_{n\ge1}\otimes^n H$ be the tensor algebra of the vector space $H$, not containing the field. Each $\otimes^n H$ is a $H$-module, by means of
\begin{align*}
h\ast(h_1\otimes\cdots\otimes h_n)=\sum(h_{(1)}h_1)\otimes\cdots\otimes(h_{(n)}h_n).
\end{align*}
Hence $T$ is an $H$-module as well. Let $X$ be a set of variables and let $T(X)$ be the vector space generated by all $tw$, where $t\in T$ and $w$ is a non-associative word. Then $T(X)$ is a $H$-module if we define $h\ast tw=(h\ast t)w$, for $h\in H$. Now, let $\mathscr{H}(X)$ be the vector subspace of $T(X)$ generated by all $tw$, with $|t|=|w|$. By \cite[Proposition 1]{BL1998}, $\mathscr{H}(X)$ has the following Universal property. If $A$ is any $H$-algebra and $\varphi:X\to A$ is any map, then there exists unique homomorphism of $H$-algebras $\bar\varphi:\mathscr{H}(X)\to A$ extending $\varphi$. Hence, one naturally defines polynomial identities of $H$-algebras, using elements of $\mathscr{H}(X)$.

Now, let $\Omega_H=\{\mu\}\cup\{\rho_h\mid h\in H\}$. If $A$ is an $H$-algebra, then the operations in $\Omega_H$ are defined on  $A$ in the following way: $\mu$ defines the original product and for each $h\in H$, $\rho_h(a):=h\ast a$. Consider the relatively free $\Omega_H$-algebra defined by the following polynomial identities:
\begin{enumerate}
\renewcommand{\labelenumi}{\roman{enumi}.}
\item $\rho_h(x)+\rho_g(x)=\rho_{h+g}(x)$,
\item $\rho_{\alpha h}(x)=\alpha\rho_h(x)$, for $\alpha\in K$,
\item $\rho_1(x)=x$ (where $1\in H$ is the unit),
\item $\rho_h(\rho_g(x))=\rho_{hg}(x)$,
\item $\rho_h(xy)=\sum (\rho_{h_{(1)}}x)(\rho_{h_{(2)}}y)$.
\end{enumerate}

A similar argument as given in Lemma \ref{translation} can be used to translate $H$-polynomial identities into polynomial identities of the relatively free $\Omega_H$-algebra satisfying identities i--v above. It is not hard to see that two $H$-algebras are isomorphic if and only if they are isomorphic as $\Omega_H$-algebras, under the operations defined above. Hence, Razmyslov's Theorem can be applied.
\begin{theorem}
Two finite-dimensional $H$-algebras over an algebraically closed field, which are simple as $H$-algebras and satisfy the same $H$-polynomial identities, are isomorphic as $H$-algebras.
\end{theorem}

\subsection{Algebras with generalized action} 
Let $\cH=(H,\Delta^{(1)}, \Delta^{(2)})$ be a triple where $H$ is a unital associative algebra and $\Delta^{(1)}, \Delta^{(2)}$ are two linear maps, called coproducts $\Delta^{(1)}, \Delta^{(2)}: H\to H\ot H$. Using Sweedler's notation, we can write $\Delta^{(i)}(h)=h^{(i)}_{(1)}\ot h^{(i)}_{(2)}$, meaning that $\Delta^{(i)}(h)$ are arbitrary tensors of degree 2. In distinction with Hopf algebras, we impose no restrictions on the coproducts. 

An algebra $A$ is called an $\cH$-algebra if $A$ is a left $H$-module via $(h,a)\to h\ast a$, for any $h\in H$ and $a\in A$ and for any $a,b\in A$, one has
\[
h\ast(ab)=(h^{(1)}_{(1)}\ast a)(h^{(1)}_{(2)}\ast b)+(h^{(2)}_{(1)}\ast b)(h^{(2)}_{(2)}\ast a)
\]
Such algebras, with a minor modification, have appeared in \cite{B1996,G2013}.  In the natural way one can define the notions of the homomorphisms of $\cH$-algebras, simple $\cH$-algebras and so on. The construction of a free $\cH$-algebra also does not create any problems (see algebras with $H$-action above) and so one can speak about $\cH$-identities. As earlier, on can define the set $\Omega_{\cH}$ consisting of one binary operation $\mu$ and the set of unary operations $\rho_h$, for each $h\in H$. As earlier, if $A$ is a $\cH$-algebra then $\rho_h(a)=h\ast a$.  The variety of $\Omega_\cH$-algebras is distinguished by the family of identical relations, one for each $h\in H$:
\[
\rho_h(\mu(x,y))=\mu(\rho_{h^{(1)}_{(1)}}(x),\rho_{h^{(1)}_{(2)}}(y))+\mu(\rho_{h^{(2)}_{(1)}}(y),\rho_{h^{(2)}_{(2)}}(x))
\]
The $\cH$-identities can be rewritten as $\Omega_{\cH}$-identities, following the approach of Lemma \ref{translation}. Skipping obvious details, we obtain one more consequence of Razmyslov's Theorem.
\begin{theorem}
Two finite-dimensional simple $\cH$-algebras over an algebraically closed field, satisfying the same $\cH$-polynomial identities, are isomorphic as $\cH$-algebras.
\end{theorem}

\ \\[0cm]
\noindent\textbf{Acknowledgments.} This work was completed while the second author visited Memorial University of Newfoundland. He appreciates its hospitality and resources. Special thanks are due to Prof. Plamen Koshlukov for suggesting the problem and pointing out good references. Both authors are grateful to Prof. Mikhail Kochetov for useful discussions.

The first author was supported by Discovery Grant 227060-2014 of the Natural Sciences and Engineering Research Council of Canada (NSERC).

The second author was supported by Fapesp, grant 2017/11.018-9.

\end{document}